\documentclass[10pt]{amsart} 
\usepackage[foot]{amsaddr}
\usepackage{amsmath,amssymb,bm, color} 

 \usepackage{float}
 
\usepackage{amsmath}
\usepackage{graphicx}

\usepackage{mathrsfs}
\usepackage{bbm}
\usepackage{bm}
\usepackage{amsfonts,amssymb}
\usepackage{multirow}
\usepackage{lineno}
\usepackage{color}

\numberwithin{equation}{section}
 \newtheorem{theorem}{Theorem}[section]
 \newtheorem{lemma}[theorem]{Lemma}

\def\3bar{{|\hspace{-.02in}|\hspace{-.02in}|}}

\def\bf{{\mathbf{f}}}
\def\bu{{\mathbf{u}}}
\def\bg{{\mathbf{g}}}
\def\bv{{\mathbf{v}}}
\def\bn{{\mathbf{n}}}

\def\beta{{\boldsymbol{\eta}}}
\def\bvarphi{{\boldsymbol{\varphi}}}

\newtheorem{algorithm}{Least Squares Weak Galerkin Algorithm}[section]

\numberwithin{equation}{section}

\def\3bar{{|\hspace{-.02in}|\hspace{-.02in}|}}

\def\ad#1{\begin{aligned}#1\end{aligned}}  \def\b#1{\mathbf{#1}} 
\def\a#1{\begin{align*}#1\end{align*}} \def\an#1{\begin{align}#1\end{align}} \def\t#1{\hbox{#1}} \def\p#1{\begin{pmatrix}#1\end{pmatrix}}

\begin{document}

\title []{A Least Squares Weak Galerkin Framework for Linear Elasticity on Polytopal Meshes}

  \author {Chunmei Wang $\dagger$ }
  \address{Department of Mathematics, University of Florida, Gainesville, FL 32611, USA. }
  \email{chunmei.wang@ufl.edu} 
 
\author {Shangyou Zhang}
\address{Department of Mathematical Sciences,  University of Delaware, Newark, DE 19716, USA}   \email{szhang@udel.edu} 

\thanks{$\dagger$  Corresponding author. }

\begin{abstract}
This paper develops and analyzes a least-squares weak Galerkin (LS-WG) finite element method for linear elasticity. By employing weak differential operators, specifically the weak gradient, weak strain tensor, and weak divergence, defined on weak finite element spaces, the proposed framework facilitates the treatment of complex boundary conditions and internal interfaces while avoiding the restrictive discrete inf-sup condition. The resulting formulation is symmetric and positive definite and exhibits robust numerical performance in the nearly incompressible regime. In addition, the proposed method offers exceptional geometric flexibility, allowing implementation on general polytopal (polygonal and polyhedral) meshes. We establish the uniqueness of the numerical solution and derive optimal-order error estimates with respect to a tailored discrete energy norm. Extensive numerical experiments confirm the theoretical convergence rates and demonstrate the method's stability, efficiency, and locking-free performance for nearly incompressible materials.
\end{abstract}


  
\maketitle

\section{Introduction}
In this work, we introduce a least squares weak Galerkin (LS-WG) finite element method for linear elasticity. Let $\Omega \subset \mathbb{R}^d$ ($d=2, 3$) denote an open, bounded, and connected domain representing an elastic body with a Lipschitz continuous boundary $\partial \Omega$. The governing equations for the displacement vector field $\mathbf{u}$, subject to an exterior force $\mathbf{f}$ and Dirichlet boundary data $\mathbf{g}$, are given by:
\begin{equation}\label{model}
    \begin{split} 
        -\nabla \cdot \sigma(\mathbf{u}) = \mathbf{f}, & \qquad\text{in } \Omega, \\
        \mathbf{u} = \mathbf{g}, & \qquad \text{on } \partial \Omega,
    \end{split}
\end{equation}
where $\sigma(\mathbf{u})$ is the symmetric Cauchy stress tensor. For isotropic, homogeneous, and linear materials, the constitutive relation is defined as:
\[
\sigma(\mathbf{u}) = 2\mu\epsilon(\mathbf{u}) + \lambda (\nabla \cdot \mathbf{u}) \mathbf{I},
\]
with the linear strain tensor $\epsilon(\mathbf{u}) = \frac{1}{2}(\nabla \mathbf{u} + \nabla \mathbf{u}^T)$. The material properties are described by the Lam\'{e} constants $\mu$ and $\lambda$. For plane strain conditions, these constants are determined by the elasticity modulus $E$ and Poisson’s ratio $\nu$ via:
\[
\lambda = \frac{E\nu}{(1+\nu)(1-2\nu)}, \quad \mu = \frac{E}{2(1+\nu)}.
\]

The model problem \eqref{model} can be equivalently reformulated to: Finding $\bu$ and $\sigma$ satisfying 
\begin{equation}\label{model2}
    \begin{split}
       -\nabla\cdot \sigma &= \bf,\qquad \text{in}\ \Omega,  \\
\sigma - 2\mu\epsilon(\bu)-\lambda (\nabla\cdot\bu)  I& =  0, \qquad \text{in}\ \Omega,  \\
  \bu&= \bg, \qquad \text{on}\ \partial \Omega.    
    \end{split}
\end{equation}

Numerical solutions for partial differential equations (PDEs) in linear elasticity rely heavily on the Finite Element Method (FEM) and its variants. Mixed FEMs are particularly prevalent, yet they face a persistent difficulty: the enforcement of strong symmetry on the stress tensor. Efforts to address this include the relaxation of symmetry constraints \cite{57}, the development of weakly symmetric elements \cite{13}, and the design of nonconforming mixed schemes \cite{4, 9, 12, 29, 62, 63, 64}. A major breakthrough facilitated the use of conforming mixed elements with reduced degrees of freedom by identifying critical structures within discrete stress spaces on simplicial grids \cite{41}.

Discontinuous Galerkin (DG) methods offer another robust framework, characterized by element-level discretization and seamless inter-element connectivity via numerical traces \cite{5, 19, 31, 59}. Notable advances in this area include hybridized mixed methods \cite{28} and the tangential-displacement normal-stress  method, the latter being particularly recognized for its locking-free properties \cite{48, 49}.

However, these traditional approaches possess inherent limitations. Standard mixed methods depend on complex $H(\text{div})$-symmetric spaces and the satisfaction of the discrete inf-sup condition. Meanwhile, DG methods typically result in a significant number of degrees of freedom and require the fine-tuning of penalty parameters to ensure stability.

In contrast, the weak Galerkin (WG) method offers a flexible alternative that relaxes regularity requiremen
ts through the use of weak derivatives and stabilization terms. By defining numerical schemes on the weak forms of the underlying PDEs, WG provides a unified and robust framework for a diverse range of computational problems \cite{wg1, wg2, wg3, wg4, wg5, wg6, wg7, wg8, wg9, wg10, wg11, wg12, wg13, wg14, wg15, wg16, wg17, wg18, wg19, wg20, wg21,   wy3655, ye, guan, guan2}.

In this work, we provide a rigorous theoretical foundation for a least squares weak Galerkin  scheme for the linear elasticity problem. The main novelty of the proposed method lies in combining a least-squares formulation of the first-order elasticity system with weak Galerkin discretization, yielding a symmetric and
positive definite (SPD) method without requiring a discrete inf-sup condition or strongly symmetric stress approximation spaces. Compared to existing methodologies in the literature, the proposed LS-WG framework offers several distinct advantages:
\begin{itemize}
  \item \textbf{Avoidance of the inf-sup condition:}
The discrete least-squares bilinear form   is SPD, and hence its matrix representation is symmetric positive
definite. Consequently, the proposed method does not require a restrictive
discrete inf-sup condition, providing greater flexibility in the choice of
finite element approximation spaces.
   \item \textbf{Elimination of strong symmetry constraints:} The use of weak differential operators (e.g., the weak gradient, weak strain tensor, and weak divergence) on discontinuous weak finite element spaces eliminates the need to construct strongly symmetric stress approximation spaces required by conforming mixed finite element methods, thereby simplifying the design of finite element spaces.
    \item \textbf{Robustness in the nearly incompressible regime:} Numerical experiments demonstrate that the proposed formulation exhibits robust performance for nearly incompressible materials. As the material approaches the incompressible limit (Poisson's ratio $\nu \to 0.5$, $\lambda \to \infty$), the method maintains optimal convergence behavior and effectively avoids the artificial stiffening commonly observed in standard displacement-based finite element methods.
    \item \textbf{Geometric flexibility:} The LS-WG scheme operates   on general polygonal and polyhedral partitions, offering a significant geometric advantage over traditional FEMs that are often strictly confined to simplicial or hexahedral grids.
\end{itemize}

We provide a complete theoretical justification for the LS-WG scheme, confirming the unique solvability of the discrete system and securing optimal-order error estimates in a designated discrete energy norm. A series of carefully designed numerical experiments is subsequently presented to validate these theoretical benchmarks, illustrating the practical efficiency, structural robustness, and overall computational viability of the methodology.

The remainder of this paper is organized as follows. In Section 2, we provide a concise review of the fundamental mathematical definitions for the weak gradient, weak strain tensor, and weak divergence, alongside their corresponding discrete counterparts. Section 3 is devoted to the development of the LS-WG formulation for the linear elasticity problem and the proof of the uniqueness of its discrete solution. Section 4 presents a rigorous derivation of the optimal-order error estimates. Finally, Section 5 presents a series of numerical experiments demonstrating the stability, accuracy, and efficiency of the proposed method, including its robust performance in the nearly incompressible regime.

Throughout this work, we adopt standard notation for Sobolev spaces and their associated norms. For any open, bounded domain $D \subset \mathbb{R}^d$ with a Lipschitz continuous boundary, $\|\cdot\|_{s,D}$ and $|\cdot|_{s,D}$ denote the norm and seminorm of the Sobolev space $H^s(D)$ for $s \ge 0$, respectively. The corresponding inner product is denoted by $(\cdot, \cdot)_{s,D}$. In the special case where $s=0$, the space $H^0(D)$ coincides with $L^2(D)$, with the norm and inner product denoted by $\|\cdot\|_D$ and $(\cdot, \cdot)_D$, respectively. For simplicity, the subscript $D$ is omitted when $D = \Omega$ or when the domain of integration is clear from context.

\section{Discrete Weak Strain Tensor and Discrete Weak Divergence}\label{Section:Hessian}
In this section, we briefly review the definitions of the weak strain tensor and weak divergence, along with their discrete counterparts, as introduced in \cite{wg10, wg18}.

Let $T$ be a polytopal element with boundary $\partial T$. A weak vector-valued function on $T$ is defined as $\bv=\{\bv_0, \bv_b\}$, where $\bv_0\in [L^2(T)]^d$ and $\bv_b\in [L^{2}(\partial T)]^d$. The first component, $\bv_0$, represents the value of $\bv$ in the interior of $T$, while the second component, $\bv_b$, represents the value of $\bv$ on the boundary of $T$. Generally, $\bv_b$ is assumed to be independent of the trace of $\bv_0$. In the special case where $\bv_b= \bv_0|_{\partial T}$, the function $\bv$ is fully determined by $\bv_0$ and is simply denoted as $\bv=\bv_0$. We denote by $W(T)$ the space of all such weak functions on $T$; i.e.,
\begin{equation*}\label{eq:WT}
 W(T)=\{\bv=\{\bv_0,\bv_b\}: \bv_0\in [L^2(T)]^d, \bv_b\in [L^{2}(\partial T)]^d\}.
\end{equation*}

Similarly, a weak tensor-valued function on $T$ is defined as $\sigma=\{\sigma_0, \sigma_b\}$, where $\sigma_0\in [L^2(T)]^{d\times d}$ and $\sigma_b\in [L^{2}(\partial T)]^{d\times d}$. By the same convention, if $\sigma_b= \sigma_0|_{\partial T}$, the function is simply denoted as $\sigma=\sigma_0$. We denote by $V(T)$ the space of all weak tensor functions on $T$; i.e.,
\begin{equation*}\label{eq:VT}
 V(T)=\{\sigma=\{\sigma_0,\sigma_b\}: \sigma_0\in [L^2(T)]^{d\times d}, \sigma_b\in [L^{2}(\partial T)]^{d\times d}\}.
\end{equation*}

The weak gradient, denoted by $\nabla_{w}$, is a linear operator from $W(T)$ to the dual space of $[H^{1}(T)]^{d\times d}$. For any $\bv\in W(T)$, the weak gradient $\nabla_w \bv$ is defined as a bounded linear functional on $[H^{1}(T)]^{d\times d}$ such that
\begin{equation*}\label{eq:weak_grad}
 (\nabla _{w}\bv, \bvarphi)_T=-(\bv_0,\nabla\cdot \bvarphi)_T+ \langle \bv_b, \bvarphi\cdot \bn \rangle_{\partial T},\quad \forall \bvarphi\in [H^{1}(T)]^{d\times d},
\end{equation*}
where $\bn$ is the unit outward normal vector to $\partial T$.
 
For any non-negative integer $r$, let $P_r(T)$ represent the space of polynomials on $T$ with a total degree of at most $r$. A discrete weak gradient on $T$, denoted by $\nabla_{w, r_1, T}$, is a linear operator from $W(T)$ to $[P_{r_1}(T)]^{d\times d}$. For any $\bv\in W(T)$, $\nabla_{w, r_1, T}\bv$ is the unique polynomial matrix in $[P_{r_1}(T)]^{d\times d}$ satisfying
\begin{equation*}\label{eq:discrete_weak_grad}
 (\nabla_{w,r_1,T} \bv, \bvarphi)_T=-(\bv_0,\nabla\cdot \bvarphi)_T+ \langle \bv_b,  \bvarphi\cdot \bn \rangle_{\partial T},\quad \forall \bvarphi \in [P_{r_1}(T)]^{d\times d}.
\end{equation*}

We define the discrete weak strain tensor as follows:
\begin{equation*}
\epsilon_{w,r_1,T}(\bu)=\frac{1}{2}(\nabla_{w,r_1,T}\bu+\nabla_{w,r_1,T}\bu^T).
\end{equation*}
 
For any $\bv\in W(T)$, the discrete weak strain tensor, denoted by $\epsilon_{w, r_1, T}(\bv)$, is the unique polynomial matrix in $[P_{r_1}(T)]^{d\times d}$ satisfying
\begin{equation}\label{2.5}
 (\epsilon_{w,r_1,T} (\bv), \bvarphi)_T=-(\bv_0,\nabla \cdot \frac{1}{2}(\bvarphi+\bvarphi^T) )_T+ \langle \bv_b,  \frac{1}{2}(\bvarphi+\bvarphi^T)\cdot \bn \rangle_{\partial T}, 
\end{equation}
for all $\bvarphi \in [P_{r_1}(T)]^{d\times d}$.
  
For a sufficiently smooth $\bv_0\in [H^1(T)]^d$, applying standard integration by parts to the first term on the right-hand side of (\ref{2.5}) gives
\begin{equation}\label{2.5new}
 (\epsilon_{w,r_1,T} (\bv), \bvarphi)_T=(\epsilon( \bv_0), \bvarphi)_T+ \langle \bv_b-\bv_0, \frac{1}{2}( \bvarphi+\bvarphi^T)\cdot \bn \rangle_{\partial T},  
\end{equation} 
for all $\bvarphi \in [P_{r_1}(T)]^{d\times d}$. 

The weak divergence of $\bv\in W(T)$, denoted by $\nabla_w\cdot \bv$, is a bounded linear functional in the Sobolev space $H^1(T)$, and its action on any $\phi\in H^1(T)$ is given by
\begin{equation*}\label{eq:weak_div_v}
    (\nabla_w\cdot \bv, \phi)_T=-(\bv_0, \nabla\phi)_T+\langle \bv_b\cdot\bn, \phi\rangle_{\partial T}.
\end{equation*} 

The discrete weak divergence of $\bv\in W(T)$, denoted by $\nabla_{w, r_2, T}\cdot \bv$, is the unique polynomial in $P_{r_2}(T)$ satisfying 
\begin{equation}\label{disdiv}
    (\nabla_{w, r_2, T}\cdot \bv, \phi)_T=-(\bv_0, \nabla\phi)_T+\langle \bv_b\cdot\bn, \phi\rangle_{\partial T},
\end{equation}
for any $\phi\in P_{r_2}(T)$. 

For a smooth $\bv_0\in [H^1(T)]^d$, applying integration by parts to the first term on the right-hand side of (\ref{disdiv}) yields
\begin{equation}\label{disdivnew}
    (\nabla_{w, r_2, T}\cdot \bv, \phi)_T= (\nabla\cdot \bv_0,  \phi)_T+\langle (\bv_b-\bv_0)\cdot\bn, \phi\rangle_{\partial T},
\end{equation}
for any $\phi\in P_{r_2}(T)$. 

Similarly, the weak divergence of a tensor $\sigma\in V(T)$, denoted by $\nabla_w\cdot \sigma$, is a bounded linear functional in the Sobolev space $[H^1(T)]^d$, and its action on any $\bv\in [H^1(T)]^d$ is given by
\begin{equation*}\label{eq:weak_div_sigma}
    (\nabla_w\cdot \sigma, \bv)_T=-(\sigma_0, \nabla\bv)_T+\langle \sigma_b\cdot\bn, \bv\rangle_{\partial T}.
\end{equation*} 

The discrete weak divergence of $\sigma\in V(T)$, denoted by $\nabla_{w, r_3, T}\cdot \sigma$, is the unique polynomial in $[P_{r_3}(T)]^d$ satisfying 
\begin{equation}\label{disdiv2}
    (\nabla_{w, r_3, T}\cdot \sigma, \bv)_T=-(\sigma_0, \nabla\bv)_T+\langle \sigma_b\cdot\bn, \bv\rangle_{\partial T},
\end{equation}
for any $\bv\in [P_{r_3}(T)]^d$. 

For a smooth $\sigma_0\in [H^1(T)]^{d\times d}$, applying integration by parts to the first term on the right-hand side of (\ref{disdiv2}) yields
\begin{equation}\label{disdivnew2}
    (\nabla_{w, r_3, T}\cdot \sigma, \bv)_T= (\nabla\cdot \sigma_0,  \bv)_T+\langle (\sigma_b-\sigma_0)\cdot\bn, \bv\rangle_{\partial T},
\end{equation}
for any $\bv\in [P_{r_3}(T)]^d$.

\section{Least Squares Weak Galerkin Algorithms}\label{Section:WGFEM}

Let $\mathcal{T}_h$ be a finite element partition of the domain $\Omega\subset \mathbb{R}^d$ into polytopal elements, where $\mathcal{T}_h$ is assumed to be shape-regular as defined in \cite{wy3655}. Denote by $\mathcal{E}_h$ the set of all edges/faces in $\mathcal{T}_h$, and let $\mathcal{E}_h^0=\mathcal{E}_h \setminus \partial\Omega$ be the set of interior edges/faces. The diameter of an element $T\in \mathcal{T}_h$ is denoted by $h_T$, and the mesh size of the partition is given by $h=\max_{T\in \mathcal{T}_h}h_T$.

For each element $T\in\mathcal{T}_h$, we define the local weak finite element space as: $k\ge 1$,
\begin{equation}\label{Wk}
W(k, T)=\{\{\bv_0,\bv_b\}: \bv_0\in [P_k(T)]^d,\bv_b\in [P_{k}(e)]^d\}.   
\end{equation}

By assembling $W(k, T)$ over all the elements $T\in \mathcal{T}_h$ and enforcing continuity on the interior interfaces $\mathcal{E}_h^0$, we define the global weak finite element space:
\begin{equation}\label{Wh}
W_h= \{\{\bv_0,\bv_b\}:\ \{\bv_0,\bv_b\}|_T\in W(k, T), \forall T\in {\mathcal T}_h, \bv_b|_{T_1}=\bv_b|_{T_2}  \ \text{on} \ e=T_1\cap T_2, \forall e\in \mathcal{E}_h^0  \}.
\end{equation}

Additionally, we denote by $W_h^0$ the subspace of $W_h$ with vanishing boundary values on $\partial\Omega$:
\begin{equation}\label{Wh0}
W_h^0=\{\{\bv_0,\bv_b\}\in W_h: \bv_b=0 \ \text{on}\ \partial\Omega\}.
\end{equation}

For each element $T\in\mathcal{T}_h$, we define a local weak finite element space  as: $k\ge 1$,
\begin{equation}\label{Vk_local}
V(k, T)=\{\{\sigma_0,\sigma_b\}: \sigma_0\in [P_{k}(T)]^{d\times d},\sigma_b\in [P_{k}(e)]^{d\times d}\}.   
\end{equation}

By assembling $V(k, T)$ over all the elements $T\in \mathcal{T}_h$ and enforcing continuity on the interior interfaces $\mathcal{E}_h^0$, we define the global weak finite element space:
\begin{equation}\label{Vh}
V_h=\big\{\{\sigma_0,\sigma_b\}:\ \{\sigma_0,\sigma_b\}|_T\in V(k, T), \forall T\in \mathcal{T}_h, \sigma_b|_{T_1}=\sigma_b|_{T_2}  \ \text{on} \ e=T_1\cap T_2, \forall e\in \mathcal{E}_h^0   \big\}.
\end{equation}
 
We take $r_1=k$, $r_2=k$ and $r_3=k$. For simplicity, the discrete weak strain tensor $\epsilon_{w, k, T}\bv$ and the discrete weak divergence $\nabla_{w, k, T} \cdot\bv$ are denoted by $\epsilon_{w}\bv$ and $\nabla_{w} \cdot\bv$, respectively. These quantities are computed locally on each element $T$ using definitions \eqref{2.5} and \eqref{disdiv}: 
\an{ \label{w-sg} 
(\epsilon_{w} \bv)|_T & = \epsilon_{w, k, T}(\bv |_T), \qquad \forall T\in \mathcal{T}_h, \\
   \label{w-d-s} 
(\nabla_{w}\cdot \bv)|_T& = \nabla_{w, k, T}\cdot(\bv |_T), \qquad \forall T\in \mathcal{T}_h. }

Similarly, the discrete weak divergence $\nabla_{w, k, T} \cdot\sigma$ is denoted simply by $\nabla_{w} \cdot\sigma$. This quantity is computed locally on each element $T$ using definition \eqref{disdiv2}: 
\an{\label{w-d-u} 
(\nabla_{w}\cdot \sigma)|_T= \nabla_{w, k, T}\cdot(\sigma |_T), \qquad \forall T\in \mathcal{T}_h. }

To enforce the connection between the interior and boundary components of the weak functions, we introduce the following stabilizers:
\[
s_1(\bu, \bv) = \sum_{T \in \mathcal{T}_h} h_T^{-1} \langle \bu_0 - \bu_b, \bv_0 - \bv_b \rangle_{\partial T}, \qquad
\forall \bu, \bv\in W_h,
\]
\[
s_2(\sigma, \delta) = \sum_{T \in \mathcal{T}_h} h_T^{-1} \langle \sigma_0 -\sigma_b, \delta_0 - \delta_b \rangle_{\partial T}, \qquad\forall \sigma, \delta\in V_h.
\]

For $\bu, \bv\in W_h$ and $\sigma,  \delta\in V_h$, we introduce the bilinear form:   
\begin{equation*} 
\begin{split}
&a ((\bu, \bv), (\sigma, \delta))=\sum_{T\in \mathcal{T}_h}(\nabla_w\cdot\sigma, \nabla_w \cdot\delta)_T\\
&+ (\sigma_0-2\mu \epsilon_w(\bu )-\lambda (\nabla_w\cdot\bu )I, \delta_0-2\mu \epsilon_w(\bv)-\lambda (\nabla_w\cdot\bv)I ) _T.
 \end{split}
\end{equation*}
The least squares WG numerical scheme for the elasticity problem \eqref{model2} is as follows:

\begin{algorithm}\label{PDWG1}
Find $\bu_h=\{\bu_0, \bu_b\} \in W_h$ and $\sigma_h=\{\sigma_0, \sigma_b\} \in V_h$ such that $\bu_b=Q_b\bg$ on $\partial\Omega$ and  
\begin{equation}\label{WG}
\begin{split}
    &a ((\bu_h, \bv), (\sigma_h, \delta))+s_1(\bu_h, \bv)+ s_2(\sigma_h, \delta) = \sum_{T\in \mathcal{T}_h}(-\bf, \nabla_w \cdot\delta)_T, 
\end{split}
\end{equation}
for all $\bv=\{\bv_0, \bv_b\}\in W_h^0$ and $\delta\in V_h$. Here $Q_b$ denotes the $L^2$ projection operator onto the space $P_k(e)$.
\end{algorithm}
 
\begin{theorem}\label{theorem1}  
The least-squares weak Galerkin scheme \eqref{WG} has a unique solution.
\end{theorem} 
\begin{proof}
It suffices to show that the solution to \eqref{WG} is trivial if $\bf=0$ and $\bg=0$. Assume that $\bf=0$ and $\bg=0$, and let $\bv=\bu_h$ and $\delta=\sigma_h$ in \eqref{WG}. This yields: 
\begin{equation*}\label{WG_proof}
\begin{split}
    & \sum_{T\in \mathcal{T}_h}(\nabla_w\cdot\sigma_h, \nabla_w \cdot\sigma_h)_T+ (\sigma_0-2\mu \epsilon_w(\bu_h)-\lambda (\nabla_w\cdot\bu_h)I,\\
    &\quad \sigma_0-2\mu \epsilon_w(\bu_h)-\lambda (\nabla_w\cdot\bu_h)I ) _T+s_1(\bu_h, \bu_h)+ s_2(\sigma_h, \sigma_h) = 0.
\end{split}
\end{equation*}
Because every term on the left-hand side is non-negative, this implies $\nabla_w\cdot\sigma_h=0$ and $\sigma_0-2\mu \epsilon_w(\bu_h)-\lambda (\nabla_w\cdot\bu_h)I=0$ on each element $T$,  $\sigma_0=\sigma_b$ on each $\partial T$.

Since $\bu_b$ and  $\sigma_b$ are single-valued on every interior interface and  $\bu_0=\bu_b$ and $\sigma_0=\sigma_b$ on each $\partial T$, the traces of $\bu_0$ and $\sigma_0$ coincide across every interior interface. Hence $\bu_0$ and $\sigma_0$ are globally continuous throughout the entire domain $\Omega$.

Using \eqref{disdivnew} and the fact that $\bu_0=\bu_b$ on $\partial T$, we obtain:
\begin{equation*} 
(\nabla_w \cdot \bu_h,\varphi)_T=(\nabla \cdot \bu_0,\varphi)_T,
\end{equation*}
for all $\varphi \in P_{k}(T)$. This establishes that $ \nabla_w \cdot \bu_h=\nabla \cdot \bu_0$ on each element $T$. 
  
Similarly, applying \eqref{2.5new} with $\bu_0=\bu_b$ on $\partial T$ gives:
\begin{equation*} 
(\epsilon_w \bu_h,\varphi)_T=(\epsilon \bu_0,\varphi)_T,
\end{equation*}
for all $\varphi \in P_{k}(T)$, which leads to $ \epsilon_w \bu_h=\epsilon \bu_0$ on each element $T$. 

Applying the analogous reasoning with \eqref{disdivnew2} and $\sigma_0=\sigma_b$ on $\partial T$ yields:
\begin{equation*} 
(\nabla_w \cdot \sigma_h,\bv)_T=(\nabla \cdot \sigma_0,\bv)_T,
\end{equation*}
for all $\bv \in [P_{k}(T)]^d$, establishing that $ \nabla_w \cdot \sigma_h=\nabla \cdot\sigma_0$ on each element $T$. 

Substituting these classical operators into our initial deductions, the conditions $\nabla_w\cdot\sigma_h=0$ and $\sigma_h-2\mu \epsilon_w(\bu_h)-\lambda (\nabla_w\cdot\bu_h)I=0$ imply that $\nabla \cdot\sigma_0=0$ and $\sigma_0-2\mu \epsilon (\bu_0)-\lambda (\nabla \cdot\bu_0)I=0$ on each element $T$. 
 
Because $\bu_0$ and $\sigma_0$ are continuous across the whole domain $\Omega$, these relations hold globally: $\nabla \cdot\sigma_0=0$ and $\sigma_0-2\mu \epsilon (\bu_0)-\lambda (\nabla \cdot\bu_0)I=0$ in $\Omega$. Furthermore, $\bu_0=\bu_b$ on $\partial T$ and the homogeneous boundary condition $\bu_b=0$ on $\partial \Omega$ together dictate that $\bu_0=0$ on $\partial \Omega$. 

From the global system $\nabla \cdot\sigma_0=0$ and $\sigma_0-2\mu \epsilon (\bu_0)-\lambda (\nabla \cdot\bu_0)I=0$ in $\Omega$, combined with $\bu_0=0$ on $\partial \Omega$, we conclude that $\bu_0 \equiv 0$ in $\Omega$. Since $\bu_0=\bu_b$ on each $\partial T$, it follows that $\bu_b\equiv 0$, which means $\bu_h \equiv 0$ in the domain $\Omega$. 

Substituting $\bu_0 \equiv 0$ into the relation $\sigma_0-2\mu \epsilon (\bu_0)-\lambda (\nabla \cdot\bu_0)I=0$ gives $\sigma_0 \equiv 0$ in $\Omega$. Again, since $\sigma_0=\sigma_b$ on each $\partial T$, we get $\sigma_b\equiv 0$, resulting in $\sigma_h \equiv 0$ in $\Omega$. 

This completes the proof of the theorem. 
\end{proof} 

We now define a semi-norm on $W_h\times V_h$ as follows:
\an{\label{three-b} \ad{
\3bar (\bu, \sigma)\3bar^2 &=  \sum_{T\in \mathcal{T}_h}(\nabla_w\cdot\sigma, \nabla_w \cdot\sigma)_T + (\sigma_0-2\mu \epsilon_w(\bu)-\lambda (\nabla_w\cdot\bu)I,\\
&\quad \  \ \ \sigma_0-2\mu \epsilon_w(\bu)-\lambda (\nabla_w\cdot\bu)I ) _T+s_1(\bu,\bu)+s_2(\sigma, \sigma) .        
 } }
 
Similar to the proof of Theorem \ref{theorem1}, we can prove that $\3bar \cdot\3bar$ defines a norm on $W_h^0\times V_h$.

\section{Error Estimates}
We begin by recalling the fundamental trace inequalities necessary for our analysis. Let $\mathcal{T}_h$ be a shape-regular finite element partition of the domain $\Omega$. For any element $T \in \mathcal{T}_h$ and function $\phi \in H^1(T)$, the following trace inequality holds \cite{wy3655}:
\begin{equation}\label{tracein}
\|\phi\|^2_{\partial T} \leq C \left( h_T^{-1} \|\phi\|_T^2 + h_T \|\nabla \phi\|_T^2 \right).
\end{equation}
When $\phi$ is a polynomial, this bound simplifies to the  trace inequality \cite{wy3655}:
\begin{equation}\label{trace}
\|\phi\|^2_{\partial T} \leq C h_T^{-1} \|\phi\|_T^2.
\end{equation}

For each element $T \in \mathcal{T}_h$, let $Q_0^k$ denote the $L^2$ projection 
onto the polynomial spaces $P_k(T)$ with $k \geq 1$. 
For each edge or face $e \subset \partial T$, let $Q_b^k$ 
  represent the $L^2$ projections onto $P_k(e)$. 

For any function $\bu \in [H^2(\Omega)]^d$, we define $Q_h^k \bu$ as the $L^2$ projection onto the weak finite element space $W_h$ such that on each element $T$:
\begin{equation*}
Q_h^k \bu = \{Q_0^k \bu, Q_b^k \bu\}.
\end{equation*}

For any function $\sigma \in [H^1(\Omega)]^{d\times d}$, we define $Q_h^{k} \sigma$ as the $L^2$ projection onto the weak finite element space $V_h$ such that on each element $T$:
\begin{equation*}
Q_h^{k} \sigma = \{Q_0^{k} \sigma, Q_b^{k} \sigma\}.
\end{equation*}

Furthermore, let $\mathcal{Q}^{k}_h$ denote the $L^2$ projections onto the spaces $P_{k}(T)$.

\begin{lemma} \cite{wy3655}  
Let $\mathcal{T}_h$ be a finite element partition of $\Omega$ satisfying the shape regularity assumptions given in \cite{wy3655}. Then, for any $0\leq s\leq 1$ and $1\leq m\leq k$, one has
\begin{align}
\label{3.1} \sum_{T\in \mathcal{T}_h}h_T^{2s}\|\bu-Q_0^k\bu\|^2_{s,T} &\leq C h^{2(m+1)}\|\bu\|_{m+1}^2, \\
\label{3.2} \sum_{T\in \mathcal{T}_h}h_T^{2s}\|\sigma-Q_0^{k}\sigma\|^2_{s,T} 
  &\leq C h^{2(m+1)}\|\sigma\|_{m+1}^2.
\end{align}
\end{lemma}

\begin{lemma} \cite{wgelas,wg10}
The projection operators $Q_h^{k}$ and  $\mathcal{Q}^{k}_h$ 
           satisfy the following commutative properties:
\begin{align}
\nabla_w\cdot(Q_h^k \bu) &= \mathcal{Q}^{k}_h(\nabla\cdot \bu), \label{EQ:CommutativeP} \\
\epsilon_w(Q_h^k \bu) &= \mathcal{Q}^{k}_h(\epsilon \bu), \label{EQ:CommutativeP2} \\
\nabla_w\cdot(Q_h^{k} \sigma) &= \mathcal{Q}^{k}_h(\nabla\cdot \sigma), \label{EQ:CommutativeP3}
\end{align}
for all $\bu\in [H^{k}(T)]^d$ and $\sigma\in [H^{k}(T)]^{d\times d}$.
\end{lemma}

\begin{theorem}
Let $\bu$ and $\sigma$ be the exact solutions to the elasticity problem \eqref{model2}, and let $\bu_h \in W_h$ and $\sigma_h\in V_h$ be the numerical solutions to the least-squares Weak Galerkin scheme \eqref{WG}. We define the error functions  
$$e_{\bu_h}=\{e_{\bu_0}, e_{\bu_b}\}=Q_h^k\bu-\bu_h, \qquad e_{\sigma_h}=\{e_{\sigma_0}, e_{\sigma_b}\}=Q_h^{k}\sigma-\sigma_h.$$ 
There exists a constant $C$ such that 
\begin{equation}\label{est1}
    \3bar (e_{\bu_h}, e_{\sigma_h})\3bar \leq C(h^{k} \|\bu\|_{k+1}+h^{k}\|\sigma\|_{k+1}).
\end{equation}
\end{theorem}
\begin{proof}
Testing the first equation in \eqref{model2} by $\nabla_w\cdot\delta$ and testing the second equation in \eqref{model2} by $\delta_0-2\mu \epsilon_w (\bv)-\lambda (\nabla_w\cdot\bv)I$ implies
\begin{equation*}
\begin{split}
\sum_{T\in \mathcal{T}_h}&(\nabla \cdot\sigma, \nabla_w\cdot\delta)_T+(\sigma-2\mu \epsilon (\bu)-\lambda (\nabla \cdot\bu)I, \delta_0-2\mu \epsilon_w (\bv)-\lambda (\nabla_w\cdot\bv)I)_T\\
&= \sum_{T\in \mathcal{T}_h}(-\mathbf{f}, \nabla_w\cdot\delta)_T.
\end{split}
\end{equation*}

Using \eqref{EQ:CommutativeP}-\eqref{EQ:CommutativeP3}, we have
\begin{equation}\label{s1}
\begin{split}
&\quad \ \sum_{T\in \mathcal{T}_h} (\nabla \cdot\sigma, \nabla_w\cdot\delta)_T+(\sigma-2\mu \epsilon (\bu)-\lambda (\nabla \cdot\bu)I, \\
 &\qquad \qquad \delta_0-2\mu \epsilon_w (\bv)-\lambda (\nabla_w\cdot\bv)I)_T\\ 
&= \sum_{T\in \mathcal{T}_h} (\mathcal{Q}_h^{k}\nabla \cdot\sigma, \nabla_w\cdot\delta)_T
   +(Q_0^{k}\sigma-\mathcal{Q}_h^{k } ( 2\mu \epsilon (\bu) \\
 &\qquad \qquad+\lambda (\nabla \cdot\bu)I), \delta_0-2\mu \epsilon_w (\bv)-\lambda (\nabla_w\cdot\bv)I)_T \\
&= \sum_{T\in \mathcal{T}_h} (\nabla_w \cdot (Q_h^{k}\sigma), \nabla_w\cdot\delta)_T
    +(Q_0^{k} \sigma-2\mu \epsilon_w (Q_h^k \bu)\\
 &\qquad \qquad -\lambda (\nabla_w \cdot (Q_h^k\bu))I, 
    \delta_0-2\mu \epsilon_w (\bv)-\lambda (\nabla_w\cdot\bv)I)_T \\
&= \sum_{T\in \mathcal{T}_h}(-\mathbf{f}, \nabla_w\cdot\delta)_T.
\end{split}
\end{equation}  

Subtracting \eqref{s1} from \eqref{WG} gives
\begin{equation*} 
\begin{split}
 \sum_{T\in \mathcal{T}_h}   ( \nabla_w \cdot (Q_h^{k}\sigma-\sigma_h), \nabla_w\cdot\delta)_T 
 + ( (Q_0^{k}\sigma-\sigma_0)  \\ -2\mu \epsilon_w (Q_h^k \bu-\bu_h)  
  -\lambda (\nabla_w \cdot (Q_h^k\bu-\bu_h))I,\  \delta_0-2\mu \epsilon_w (\bv)
      \\ -\lambda (\nabla_w\cdot\bv)I)_T    - s_1(\bu_h, \bv) - s_2(\sigma_h, \delta)  = 0.
\end{split}
\end{equation*}  

Substituting the error functions and rearranging yields:
\begin{equation}\label{ss}
\begin{split}
&\sum_{T\in \mathcal{T}_h}  ( \nabla_w \cdot (e_{\sigma_h}), \nabla_w\cdot\delta)_T + (e_{\sigma_0}-2\mu \epsilon_w (e_{\bu_h})-\lambda (\nabla_w \cdot e_{\bu_h})I, \\& \qquad\delta_0-2\mu \epsilon_w (\bv)-\lambda (\nabla_w\cdot\bv)I)_T    + s_1(e_{\bu_h}, \bv) + s_2(e_{\sigma_h}, \delta) \\
&  = s_1(Q_h^k\bu, \bv) + s_2(Q_h^{k}\sigma, \delta).
\end{split}
\end{equation} 

Letting $\bv=e_{\bu_h}$ and $\delta=e_{\sigma_h}$ in \eqref{ss} yields
\begin{equation}\label{sss}
\begin{split}
\3bar (e_{\bu_h}, e_{\sigma_h})\3bar^2 = s_1(Q_h^k\bu, e_{\bu_h}) + s_2(Q_h^{k}\sigma, e_{\sigma_h}).
\end{split}
\end{equation}  

Using the trace inequality \eqref{tracein} and the estimate \eqref{3.1} with $s=0, 1$ and $m=k$, we have
\begin{equation}\label{a1}
\begin{split} 
&\quad \ s_1(Q_h^k\bu, e_{\bu_h}) \\
&= \sum_{T\in \mathcal{T}_h} h_T^{-1}\langle Q_0^k\bu-Q_b^k\bu, e_{\bu_0}-e_{\bu_b}\rangle_{\partial T}\\
&\leq \left(\sum_{T\in \mathcal{T}_h}h_T^{-1}\|Q_0^k\bu-Q_b^k\bu\|^2_{\partial T}\right)^{\frac{1}{2}} \left(\sum_{T\in \mathcal{T}_h}h_T^{-1}\|e_{\bu_0}-e_{\bu_b}\|^2_{\partial T}\right)^{\frac{1}{2}}\\
&\leq C \left(\sum_{T\in \mathcal{T}_h}h_T^{-2}\|Q_0^k\bu- \bu\|^2_T 
   +  |Q_0^k\bu- \bu |^2_{1, T}\right)^{\frac{1}{2}}\3bar (e_{\bu_h}, e_{\sigma_h})\3bar\\
&\leq C h^k\|\bu\|_{k+1}\3bar (e_{\bu_h}, e_{\sigma_h})\3bar.
\end{split}
\end{equation}

Similarly, using the trace inequality \eqref{tracein} and the estimate \eqref{3.2} with $s=0, 1$ and $m=k$, we obtain:
\begin{equation}\label{a2}
\begin{split}
&\quad \ s_2(Q_h^{k}\sigma, e_{\sigma_h})\\
 &= \sum_{T\in \mathcal{T}_h} h_T^{-1}\langle Q_0^{k}\sigma-Q_b^{k}\sigma, e_{\sigma_0}-e_{\sigma_b}\rangle_{\partial T}\\
&\leq \left(\sum_{T\in \mathcal{T}_h}h_T^{-1}\|Q_0^{k}\sigma-Q_b^{k}\sigma\|^2_{\partial T}\right)^{\frac{1}{2}} \left(\sum_{T\in \mathcal{T}_h}h_T^{-1}\|e_{\sigma_0}-e_{\sigma_b}\|^2_{\partial T}\right)^{\frac{1}{2}}\\
&\leq C \left(\sum_{T\in \mathcal{T}_h}h_T^{-2}\|Q_0^{k}\sigma- \sigma\|^2_T + 
  |Q_0^{k}\sigma- \sigma |^2_{1, T}\right)^{\frac{1}{2}}\3bar (e_{\bu_h}, e_{\sigma_h})\3bar\\
&\leq C h^{k }\|\sigma\|_{k+1}\3bar (e_{\bu_h}, e_{\sigma_h})\3bar.
\end{split}
\end{equation}

Substituting \eqref{a1} and \eqref{a2} into \eqref{sss} completes the proof of the theorem. 
\end{proof}

\section{Numerical Tests}

In the numerical test,  we solve the linear elasticity equations \eqref{model2} 
   on the unit square domain $\Omega=(0,1)\times(0,1)$, where  
\an{\label{s2}
   \b u =\p{ \ 2^7 (2y^3 - 3y^2 + y) (x-x^2)^2 \\ -2^7 (2x^3 - 3x^2 + x) (y-y^2)^2}, \ \ \
     \mu=1, \ \ \ \lambda=1, \ 10^2 \ \t{or} \ 10^5, }  
and $\b g$, $\sigma$ and $\b f$ are defined by $\b u$.

\begin{figure}[H]
\begin{center}\setlength\unitlength{2.4pt}\centering 
 \begin{picture}(140,45)(0,0) \put(0,41){$G_1:$}  \put(50,41){$G_2:$} \put(100,41){$G_3:$} 
  
\def\sq{\begin{picture}(40,40)(0,0) 
  \multiput(0,0)(40,0){2}{\line(0,1){40}}\multiput(0,0)(0,40){2}{\line(1,0){40}} \end{picture} }
  
\put(0,0){\begin{picture}(40,40)(0,0)
  \multiput(0,0)(0,40){1}{\multiput(0,0)(40,0){1}{\sq}} 
  \end{picture} }
  
\put(50,0){\setlength\unitlength{1.2pt}\begin{picture}(40,40)(0,0)
  \multiput(0,0)(0,40){2}{\multiput(0,0)(40,0){2}{\sq}} 
  \end{picture} } 
\put(100,0){\setlength\unitlength{0.6pt}\begin{picture}(40,40)(0,0)
  \multiput(0,0)(0,40){4}{\multiput(0,0)(40,0){4}{\sq}} 
  \end{picture} } 
\end{picture}\end{center}
\caption{The uniform square grids used in Tables \ref{t1}--\ref{t4}. }
\label{f-1}
\end{figure}
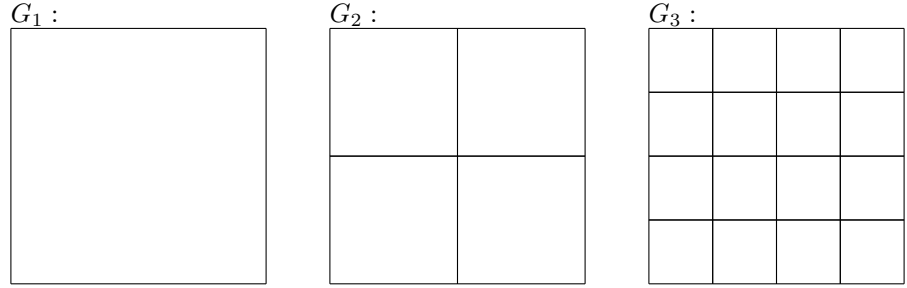

 We apply the weak Galerkin $P_k^2$-$P_k^{2\times 2}$/$P_{k}^{2\times 2}$-$P_{k}$-$P_{k}^2$ 
    finite element methods with
      $k=1,2,3$ and $4$.
That is,  $\b v_0 \in P_k(T)^2$ and $\b v_b \in P_k(e)^2$ 
      in \eqref{Wk};
      $\sigma_0 \in P_k(T)^{2\times 2}$ and $\sigma_b \in P_k(e)^{2\times 2}$ 
        in \eqref{Vh};
      $\epsilon_w \b v_h \in P_{k}(T)^{2\times 2}$  in \eqref{w-sg};
      $\nabla_w \cdot \b v_h \in P_{k}(T)$  in \eqref{w-d-s};
       $\nabla_w \cdot \sigma_h \in P_{k}(T)^2$  in \eqref{w-d-u}.  
       
We test the computation on three types of grids, shown in Figures \ref{f-1}--\ref{f-3}.
The results are listed in Tables \ref{t1}-\ref{t4}.
  In these tables, $G_i$ denotes the $i$-th grid of each type,
  and 
\a{ \3bar E_h\3bar=\3bar (\b u-\b u_h, \sigma-\sigma_h) \3bar,
 } where $\3bar \cdot \3bar$ is defined in \eqref{three-b}.

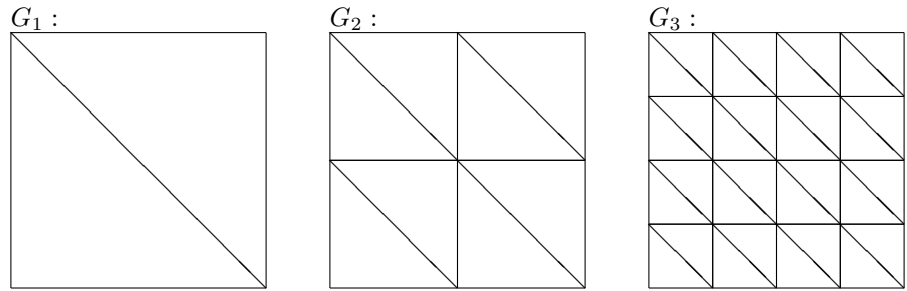
\begin{figure}[H]
\begin{center}\setlength\unitlength{2.4pt}\centering 
 \begin{picture}(140,45)(0,0) \put(0,41){$G_1:$}  \put(50,41){$G_2:$} \put(100,41){$G_3:$} 
  
\def\sq{\begin{picture}(40,40)(0,0) \put(0,40){\line(1,-1){40}}
  \multiput(0,0)(40,0){2}{\line(0,1){40}}\multiput(0,0)(0,40){2}{\line(1,0){40}} \end{picture} }
  
\put(0,0){\begin{picture}(40,40)(0,0)
  \multiput(0,0)(0,40){1}{\multiput(0,0)(40,0){1}{\sq}} 
  \end{picture} }
  
\put(50,0){\setlength\unitlength{1.2pt}\begin{picture}(40,40)(0,0)
  \multiput(0,0)(0,40){2}{\multiput(0,0)(40,0){2}{\sq}} 
  \end{picture} } 
\put(100,0){\setlength\unitlength{0.6pt}\begin{picture}(40,40)(0,0)
  \multiput(0,0)(0,40){4}{\multiput(0,0)(40,0){4}{\sq}} 
  \end{picture} } 
\end{picture}\end{center}
\caption{The triangular  grids used in Tables \ref{t1}--\ref{t4}. }
\label{f-2}
\end{figure}

\begin{table}[H]
  \centering  \renewcommand{\arraystretch}{1.1}
  \caption{Error profile by the $P_1$ WG element  for computing \eqref{s2}. }
  \label{t1}
\begin{tabular}{c|cc|cc|cc}
\hline
 $G_i$ &   $\| u-u_h\|_{0}$ & $O(h^r)$ & $\3bar E_h\3bar  $& $O(h^r)$  
    & $\|\sigma-\sigma_h\|_0  $& $O(h^r)$ \\ \hline
    &  \multicolumn{6}{c}{On square meshes (Figure \ref{f-1}), $\lambda=1$}    \\
\hline  
 5&    0.542E-01 &  0.9&    0.330E+01 &  1.0&    0.519E+00 &  0.6 \\
 6&    0.235E-01 &  1.2&    0.164E+01 &  1.0&    0.287E+00 &  0.9 \\
 7&    0.840E-02 &  1.5&    0.808E+00 &  1.0&    0.135E+00 &  1.1 \\
\hline  &  \multicolumn{6}{c}{On square meshes (Figure \ref{f-1}), $\lambda=10^5$}    \\
\hline  
 5&    0.502E-01 &  1.2&    0.265E+03 &  1.9&    0.584E+00 &  0.5 \\
 6&    0.204E-01 &  1.3&    0.677E+02 &  2.0&    0.350E+00 &  0.7 \\
 7&    0.700E-02 &  1.5&    0.170E+02 &  2.0&    0.159E+00 &  1.1 \\
\hline 
    &  \multicolumn{6}{c}{On triangular meshes (Figure \ref{f-2}), $\lambda=1$}    \\  
\hline  
 5&    0.327E-01 &  1.3&    0.135E+01 &  1.0&    0.327E+00 &  1.0 \\
 6&    0.112E-01 &  1.5&    0.666E+00 &  1.0&    0.144E+00 &  1.2 \\
 7&    0.344E-02 &  1.7&    0.327E+00 &  1.0&    0.597E-01 &  1.3 \\
\hline 
 &  \multicolumn{6}{c}{On triangular meshes (Figure \ref{f-2}), $\lambda=10^5$}    \\  
\hline   
 4&    0.805E-01 &  1.2&    0.113E+05 &  0.6&    0.668E+00 &  0.7 \\
 5&    0.314E-01 &  1.4&    0.604E+04 &  0.9&    0.366E+00 &  0.9 \\
 6&    0.108E-01 &  1.5&    0.307E+04 &  1.0&    0.167E+00 &  1.1 \\
\hline 
    &  \multicolumn{6}{c}{ On  polygonal meshes (Figure \ref{f-3}),  $\lambda=1$}    \\  
\hline   
 4&    0.936E-01 &  1.0&    0.760E+01 &  0.8&    0.722E+00 &  0.5 \\
 5&    0.459E-01 &  1.0&    0.390E+01 &  1.0&    0.455E+00 &  0.7 \\
 6&    0.184E-01 &  1.3&    0.195E+01 &  1.0&    0.236E+00 &  0.9 \\
\hline 
    &  \multicolumn{6}{c}{ On  polygonal meshes (Figure \ref{f-3}),  $\lambda=10^5$}    \\  
\hline   
 4&    0.990E-01 &  1.3&    0.525E+04 &  0.9&    0.782E+00 &  0.6 \\
 5&    0.418E-01 &  1.2&    0.265E+04 &  1.0&    0.525E+00 &  0.6 \\
 6&    0.157E-01 &  1.4&    0.133E+04 &  1.0&    0.290E+00 &  0.9 \\
\hline 
    \end{tabular}%
\end{table}%

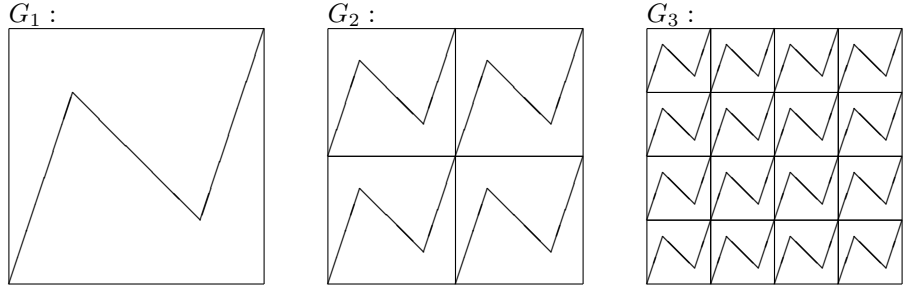
\begin{figure}[H]
\begin{center}\setlength\unitlength{2.4pt}\centering 
 \begin{picture}(140,45)(0,0) \put(0,41){$G_1:$}  \put(50,41){$G_2:$} \put(100,41){$G_3:$} 
  
\def\sq{\begin{picture}(40,40)(0,0) \put(0,0){\line(1,3){10}}  \put(40,40){\line(-1,-3){10}} \put(10,30){\line(1,-1){20}}
  \multiput(0,0)(40,0){2}{\line(0,1){40}}\multiput(0,0)(0,40){2}{\line(1,0){40}} \end{picture} }
  
\put(0,0){\begin{picture}(40,40)(0,0)
  \multiput(0,0)(0,40){1}{\multiput(0,0)(40,0){1}{\sq}} 
  \end{picture} }
  
\put(50,0){\setlength\unitlength{1.2pt}\begin{picture}(40,40)(0,0)
  \multiput(0,0)(0,40){2}{\multiput(0,0)(40,0){2}{\sq}} 
  \end{picture} } 

\put(100,0){\setlength\unitlength{0.6pt}\begin{picture}(40,40)(0,0)
  \multiput(0,0)(0,40){4}{\multiput(0,0)(40,0){4}{\sq}} 
  \end{picture} } 
\end{picture}\end{center}
\caption{The non-convex polygonal grids used in Tables \ref{t1}--\ref{t4}. }
\label{f-3}
\end{figure}

\begin{table}[H]
  \centering  \renewcommand{\arraystretch}{1.1}
  \caption{Error profile by the $P_2$ WG element  for computing \eqref{s2}. }
  \label{t2}
\begin{tabular}{c|cc|cc|cc}
\hline
 $G_i$ &   $\| u-u_h\|_{0}$ & $O(h^r)$ & $\3bar E_h\3bar  $& $O(h^r)$  
    & $\|\sigma-\sigma_h\|_0  $& $O(h^r)$ \\ \hline
    &  \multicolumn{6}{c}{On square meshes (Figure \ref{f-1}), $\lambda=1$}    \\
\hline  
 4&    0.103E-01 &  2.8&    0.211E+01 &  1.8&    0.123E+00 &  2.4 \\
 5&    0.129E-02 &  3.0&    0.542E+00 &  2.0&    0.238E-01 &  2.4 \\
 6&    0.154E-03 &  3.1&    0.136E+00 &  2.0&    0.492E-02 &  2.3 \\
\hline  &  \multicolumn{6}{c}{On square meshes (Figure \ref{f-1}), $\lambda=10^5$}    \\
\hline  
 4&    0.195E-02 &  3.1&    0.469E+00 &  1.9&    0.711E-01 &  2.0 \\
 5&    0.244E-03 &  3.0&    0.120E+00 &  2.0&    0.184E-01 &  2.0 \\
 6&    0.310E-04 &  3.0&    0.303E-01 &  2.0&    0.466E-02 &  2.0 \\
\hline 
    &  \multicolumn{6}{c}{On triangular meshes (Figure \ref{f-2}), $\lambda=1$}    \\  
\hline  
 4&    0.195E-02 &  3.1&    0.469E+00 &  1.9&    0.711E-01 &  2.0 \\
 5&    0.244E-03 &  3.0&    0.120E+00 &  2.0&    0.184E-01 &  2.0 \\
 6&    0.310E-04 &  3.0&    0.303E-01 &  2.0&    0.466E-02 &  2.0 \\
\hline 
 &  \multicolumn{6}{c}{On triangular meshes (Figure \ref{f-2}), $\lambda=10^5$}    \\  
\hline   
 3&    0.166E-01 &  2.9&    0.224E+05 &  1.6&    0.302E+00 &  1.7 \\
 4&    0.200E-02 &  3.0&    0.624E+04 &  1.8&    0.743E-01 &  2.0 \\
 5&    0.259E-03 &  3.0&    0.160E+04 &  2.0&    0.189E-01 &  2.0 \\
\hline 
    &  \multicolumn{6}{c}{ On  polygonal meshes (Figure \ref{f-3}),  $\lambda=1$}    \\  
\hline   
 3&    0.504E-01 &  2.8&    0.786E+01 &  1.1&    0.460E+00 &  1.8 \\
 4&    0.696E-02 &  2.9&    0.222E+01 &  1.8&    0.942E-01 &  2.3 \\
 5&    0.864E-03 &  3.0&    0.572E+00 &  2.0&    0.197E-01 &  2.3 \\
\hline 
    &  \multicolumn{6}{c}{ On  polygonal meshes (Figure \ref{f-3}),  $\lambda=10^5$}    \\  
\hline   
 3&    0.502E-01 &  2.9&    0.322E+05 &  1.9&    0.507E+00 &  1.6 \\
 4&    0.693E-02 &  2.9&    0.854E+04 &  1.9&    0.103E+00 &  2.3 \\
 5&    0.837E-03 &  3.1&    0.217E+04 &  2.0&    0.204E-01 &  2.3 \\
\hline 
    \end{tabular}%
\end{table}%

In Table \ref{t1}, when the \(P_{1}\) WG method is applied, 
 the error of the triple-bar norm converges at order one, as proved theoretically. 
When \(\lambda=10^5\), the large triple-bar norm is caused by the scaling of the norm. 
The \(L^{2}\) errors indicate that the method is locking-free. 
These results are consistently shown in Tables \ref{t2}--\ref{t4}. 
For the \(P_{3}\) and \(P_{4}\) methods, since the condition number is too large
  --- leading to a severe loss of significant digits
  --- we let \(\lambda=10^2\) instead of \(\lambda=10^5\).

\begin{table}[H]
  \centering  \renewcommand{\arraystretch}{1.1}
  \caption{Error profile by the $P_3$ WG element  for computing \eqref{s2}. }
  \label{t3}
\begin{tabular}{c|cc|cc|cc}
\hline
 $G_i$ &   $\| u-u_h\|_{0}$ & $O(h^r)$ & $\3bar E_h\3bar  $& $O(h^r)$  
    & $\|\sigma-\sigma_h\|_0  $& $O(h^r)$ \\ \hline
    &  \multicolumn{6}{c}{On square meshes (Figure \ref{f-1}), $\lambda=1$}    \\
\hline  
 4&    0.201E-02 &  3.9&    0.694E+00 &  2.9&    0.192E-01 &  3.7 \\
 5&    0.128E-03 &  4.0&    0.877E-01 &  3.0&    0.152E-02 &  3.7 \\
 6&    0.799E-05 &  4.0&    0.110E-01 &  3.0&    0.142E-03 &  3.4 \\
\hline  &  \multicolumn{6}{c}{On square meshes (Figure \ref{f-1}), $\lambda=10^2$}    \\
\hline  
 3&    0.295E-01 &  3.4&    0.659E+02 &  2.3&    0.263E+00 &  3.7 \\
 4&    0.202E-02 &  3.9&    0.903E+01 &  2.9&    0.204E-01 &  3.7 \\
 5&    0.128E-03 &  4.0&    0.115E+01 &  3.0&    0.174E-02 &  3.6 \\
\hline 
    &  \multicolumn{6}{c}{On triangular meshes (Figure \ref{f-2}), $\lambda=1$}    \\  
\hline  
 3&    0.239E-02 &  3.5&    0.484E+00 &  2.6&    0.501E-01 &  2.8 \\
 4&    0.165E-03 &  3.9&    0.641E-01 &  2.9&    0.563E-02 &  3.2 \\
 5&    0.106E-04 &  4.0&    0.812E-02 &  3.0&    0.624E-03 &  3.2 \\
\hline 
 &  \multicolumn{6}{c}{On triangular meshes (Figure \ref{f-2}), $\lambda=10^2$}    \\  
\hline   
 3&    0.241E-02 &  3.5&    0.820E+01 &  2.0&    0.555E-01 &  2.8 \\
 4&    0.168E-03 &  3.8&    0.117E+01 &  2.8&    0.657E-02 &  3.1 \\
 5&    0.107E-04 &  4.0&    0.150E+00 &  3.0&    0.746E-03 &  3.1 \\
\hline 
    &  \multicolumn{6}{c}{ On  polygonal meshes (Figure \ref{f-3}),  $\lambda=1$}    \\  
\hline   
 3&    0.160E-01 &  3.4&    0.423E+01 &  2.7&    0.136E+00 &  3.6 \\
 4&    0.110E-02 &  3.9&    0.557E+00 &  2.9&    0.112E-01 &  3.6 \\
 5&    0.699E-04 &  4.0&    0.704E-01 &  3.0&    0.100E-02 &  3.5 \\ 
\hline 
    &  \multicolumn{6}{c}{ On  polygonal meshes (Figure \ref{f-3}),  $\lambda=10^2$}    \\  
\hline   
 3&    0.160E-01 &  3.4&    0.354E+02 &  2.4&    0.141E+00 &  3.6 \\
 4&    0.110E-02 &  3.9&    0.484E+01 &  2.9&    0.124E-01 &  3.5 \\
 5&    0.698E-04 &  4.0&    0.618E+00 &  3.0&    0.115E-02 &  3.4 \\
\hline 
    \end{tabular}%
\end{table}%

\begin{table}[H]
  \centering  \renewcommand{\arraystretch}{1.1}
  \caption{Error profile by the $P_4$ WG element  for computing \eqref{s2}. }
  \label{t4}
\begin{tabular}{c|cc|cc|cc}
\hline
 $G_i$ &   $\| u-u_h\|_{0}$ & $O(h^r)$ & $\3bar E_h\3bar  $& $O(h^r)$  
    & $\|\sigma-\sigma_h\|_0  $& $O(h^r)$ \\ \hline
    &  \multicolumn{6}{c}{On square meshes (Figure \ref{f-1}), $\lambda=1$}    \\
\hline  
 2&    0.316E+00 &  1.8&    0.416E+02 &  3.4&    0.169E+01 &  4.7 \\
 3&    0.112E-01 &  4.8&    0.279E+01 &  3.9&    0.555E-01 &  4.9 \\
 4&    0.361E-03 &  5.0&    0.178E+00 &  4.0&    0.186E-02 &  4.9 \\
\hline  &  \multicolumn{6}{c}{On square meshes (Figure \ref{f-1}), $\lambda=10^2$}    \\
\hline  
 2&    0.312E+00 &  1.9&    0.638E+03 &  0.0&    0.173E+01 &  4.6 \\
 3&    0.112E-01 &  4.8&    0.458E+02 &  3.8&    0.586E-01 &  4.9 \\
 4&    0.364E-03 &  4.9&    0.295E+01 &  4.0&    0.212E-02 &  4.8 \\
\hline 
    &  \multicolumn{6}{c}{On triangular meshes (Figure \ref{f-2}), $\lambda=1$}    \\  
\hline  
 3&    0.387E-03 &  4.6&    0.855E-01 &  3.8&    0.719E-02 &  3.9 \\
 4&    0.130E-04 &  4.9&    0.551E-02 &  4.0&    0.482E-03 &  3.9 \\
 5&    0.412E-06 &  5.0&    0.347E-03 &  4.0&    0.314E-04 &  3.9 \\
\hline 
 &  \multicolumn{6}{c}{On triangular meshes (Figure \ref{f-2}), $\lambda=10^2$}    \\  
\hline   
 3&    0.393E-03 &  4.6&    0.252E+01 &  3.6&    0.794E-02 &  3.9 \\
 4&    0.132E-04 &  4.9&    0.167E+00 &  3.9&    0.521E-03 &  3.9 \\
 5&    0.420E-06 &  5.0&    0.106E-01 &  4.0&    0.337E-04 &  4.0 \\
\hline 
    &  \multicolumn{6}{c}{ On  polygonal meshes (Figure \ref{f-3}),  $\lambda=1$}    \\  
\hline   
 2&    0.158E+00 &  2.7&    0.226E+02 &  3.3&    0.841E+00 &  4.5 \\
 3&    0.586E-02 &  4.8&    0.153E+01 &  3.9&    0.283E-01 &  4.9 \\
 4&    0.190E-03 &  4.9&    0.975E-01 &  4.0&    0.919E-03 &  4.9 \\
\hline 
    &  \multicolumn{6}{c}{ On  polygonal meshes (Figure \ref{f-3}),  $\lambda=10^2$}    \\  
\hline   
 2&    0.158E+00 &  2.7&    0.476E+03 &  1.6&    0.844E+00 &  4.5 \\
 3&    0.586E-02 &  4.8&    0.348E+02 &  3.8&    0.289E-01 &  4.9 \\
 4&    0.191E-03 &  4.9&    0.226E+01 &  3.9&    0.138E-02 &  4.4 \\
\hline 
    \end{tabular}%
\end{table}%


\begin{thebibliography}{99}

\bibitem{4}{\sc Arnold, D.N.,Awanou, G.,Winther, R},  {\em Nonconforming tetrahedral mixed
finite elements for elasticity}. Mathematical Models and Methods in Applied
Sciences 24(04), 783–796 (2014).

\bibitem{5}{\sc Arnold, D.N., Brezzi, F., Cockburn, B., Marini, L.D},  {\em Unified analysis of
discontinuous Galerkin methods for elliptic problems}. SIAM journal on
numerical analysis 39(5), 1749–1779 (2002).

\bibitem{9}{\sc  Arnold, D.N., Winther, R},  {\em Nonconforming mixed elements for elasticity.
Mathematical models and methods in applied sciences}, 13(03), 295–307(2003).

\bibitem{12}{\sc  Awanou, G},  {\em A rotated nonconforming rectangular mixed element for elasticity}. Calcolo 46(1), 49–60 (2009).

\bibitem{13}{\sc  Boffi, D., Brezzi, F., Fortin, M},  {\em Reduced symmetry elements in linear
elasticity}. Communications on Pure \& Applied Analysis 8(1), 95–121
(2009). 

   \bibitem{wg11}{\sc  S. Cao, C. Wang and J. Wang},  {\em A new numerical method for div-curl Systems with Low Regularity Assumptions}, Computers and Mathematics with Applications, vol. 144, pp. 47-59, 2022. 


 

   \bibitem{19}{\sc  Chen, Y., Huang, J., Huang, X., Xu, Y},  {\em On the local discontinuous Galerkin
method for linear elasticity}. Mathematical Problems in Engineering (2010).



 \bibitem{guan} {\sc Q. Guan, M. Gunzburger, W. Zhao},  {\em Weak-Galerkin Finite Element Methods for a Second-Order Elliptic Variational Inequality},
Computer Methods in Applied Mechanics and Engineering 337, 677-688.

 \bibitem{guan2} {\sc Q. Guan, W. Zhao},  {\em Error Analysis of Energy Stable Weak Galerkin Schemes for the Allen-Cahn equation}, 
 Advances in Computational Science \& Engineering (ACSE) 7, 1. 
 
\bibitem{28}{\sc Gong, S., Wu, S., Xu, J.},  {\em  New hybridized mixed methods for linear elasticity
and optimal multilevel solvers}. Numerische Mathematik 141(2),
569–604 (2019).

\bibitem{29}{\sc  Gopalakrishnan, J., Guzman, J},  {\em Symmetric nonconforming mixed finite
elements for linear elasticity}. SIAM Journal on Numerical Analysis 49(4),
1504–1520 (2011).

\bibitem{31}{\sc  Hong, Q., Hu, J., Shu, S., Xu, J},  {\em A discontinuous Galerkin method for the
fourth-order curl problem}. Journal of Computational Mathematics 30(6),
565–578 (2012).
 
\bibitem{41}{\sc Hu, J},  {\em Finite element approximations of symmetric tensors on simplicial
grids inRn: The higher order case}. Journal of Computational Mathematics
33(3), 283–296 (2015)


 

 \bibitem{wg14}{\sc D. Li, Y. Nie, and C. Wang},  {\em Superconvergence of Numerical Gradient for Weak Galerkin Finite Element Methods on Nonuniform Cartesian Partitions in Three Dimensions}, Computers and Mathematics with Applications, vol 78(3), pp. 905-928, 2019.
 
  \bibitem{wg1} {\sc D. Li, C. Wang and J. Wang},  {\em An Extension of the Morley Element on General Polytopal Partitions Using Weak Galerkin Methods}, Journal of Scientific Computing, 100, vol 27, 2024.  
  
 \bibitem{wg2} {\sc D. Li, C. Wang and S. Zhang},  {\em Weak Galerkin methods for elliptic interface problems on curved polygonal partitions}, Journal of Computational and Applied Mathematics, pp. 115995, 2024. 
 
\bibitem{wg5} {\sc D. Li, C. Wang, J.  Wang and X. Ye},  {\em Generalized weak Galerkin finite element methods for second order elliptic problems}, Journal of Computational and Applied Mathematics, vol. 445, pp. 115833, 2024.

 \bibitem{wg6} {\sc D. Li, C. Wang, J. Wang and S. Zhang},  {\em High Order Morley Elements for Biharmonic Equations on Polytopal Partitions}, Journal of Computational and Applied Mathematics, Vol. 443, pp. 115757, 2024.
 
 \bibitem{wg7} {\sc D. Li, C. Wang and J. Wang},  {\em Curved Elements in Weak Galerkin Finite Element Methods}, Computers and Mathematics with Applications, Vol. 153, pp. 20-32, 2024.
 
\bibitem{wg8} {\sc D. Li, C. Wang and J. Wang},  {\em Generalized Weak Galerkin Finite Element Methods for Biharmonic Equations}, Journal of Computational and Applied Mathematics, vol. 434, 115353, 2023.
  
  \bibitem{wg13}{\sc  D. Li, C. Wang, and J. Wang},  {\em Superconvergence of the Gradient Approximation for Weak Galerkin Finite Element Methods on Rectangular Partitions}, Applied Numerical Mathematics, vol. 150, pp. 396-417, 2020.
%

\bibitem{48}{\sc  Pechstein, A., Schoberl, J},  {\em Tangential-displacement and normal–normalstress
continuous mixed finite elements for elasticity}. Mathematical Models
and Methods in Applied Sciences 21(08), 1761–1782 (2011)

\bibitem{49}{\sc Pechstein, A.S., Sch oberl, J},  {\em An analysis of the TDNNS method using
natural norms}. Numerische Mathematik 139(1), 93–120 (2018)

\bibitem{57}{\sc  Fraeijs de Veubeke, B},  {\em Stress function approach. Proc. of the World
Congress on Finite Element Methods in Structural Mechanics}. Vol. 1,
Bournemouth, Dorset, England pp. J.1–J.51 (1975)
%
%
%
   \bibitem{wg15}{\sc C. Wang},  {\em New Discretization Schemes for Time-Harmonic Maxwell Equations by Weak Galerkin Finite Element Methods}, Journal of Computational and Applied Mathematics, Vol. 341, pp. 127-143, 2018.  
    
 \bibitem{wg17}{\sc C. Wang and J. Wang},  {\em Discretization of Div-Curl Systems by Weak Galerkin Finite Element Methods on Polyhedral Partitions}, Journal of Scientific Computing, Vol. 68, pp. 1144-1171, 2016. 
 
   \bibitem{wg19}{\sc C. Wang and J. Wang},  {\em A Hybridized Formulation for Weak Galerkin Finite Element Methods for Biharmonic Equation on Polygonal or Polyhedral Meshes}, International Journal of Numerical Analysis and Modeling, Vol. 12, pp. 302-317, 2015. 
   
 \bibitem{wg20}{\sc  J. Wang and C. Wang},  {\em Weak Galerkin Finite Element Methods for Elliptic PDEs}, Science China, Vol. 45, pp. 1061-1092, 2015.  
 
 \bibitem{wg21}{\sc C. Wang and J. Wang},  {\em An Efficient Numerical Scheme for the Biharmonic Equation by Weak Galerkin Finite Element Methods on Polygonal or Polyhedral Meshes}, Journal of Computers and Mathematics with Applications, Vol. 68, 12, pp. 2314-2330, 2014.  
 
   \bibitem{wg18}{\sc C. Wang, J. Wang, R. Wang and R. Zhang},  {\em A Locking-Free Weak Galerkin Finite Element Method for Elasticity Problems in the Primal Formulation}, Journal of Computational and Applied Mathematics, Vol. 307, pp. 346-366, 2016.   
 

 \bibitem{wg12}{\sc  C. Wang, J. Wang, X. Ye and S. Zhang},  {\em De Rham Complexes for Weak Galerkin Finite Element Spaces}, Journal of Computational and Applied Mathematics, vol. 397, pp. 113645, 2021.
 
 \bibitem{wg3} {\sc C. Wang, J. Wang and S. Zhang},  {\em Weak Galerkin Finite Element Methods for Optimal Control Problems Governed by Second Order Elliptic Partial Differential Equations}, Journal of Computational and Applied Mathematics, in press, 2024. 
   
 \bibitem{wg9} {\sc C. Wang, J. Wang and S. Zhang},  {\em Weak Galerkin Finite Element Methods for Quad-Curl Problems}, Journal of Computational and Applied Mathematics, vol. 428, pp. 115186, 2023.
 
   \bibitem{wy3655} {\sc J. Wang, and X. Ye}, {\em A weak Galerkin mixed finite element method for second-order elliptic problems}, Math. Comp., vol. 83, pp. 2101-2126, 2014.
   
\bibitem{59}{\sc Wang, F., Wu, S., Xu, J},  {\em A mixed discontinuous Galerkin method for
linear elasticity with strongly imposed symmetry}. Journal of Scientific
Computing 83(1), 1–17 (2020).

\bibitem{62}{\sc  Wu, S., Gong, S., Xu, J},  {\em Interior penalty mixed finite element methods of
any order in any dimension for linear elasticity with strongly symmetric
stress tensor}. Mathematical Models and Methods in Applied Sciences
27(14), 2711–2743 (2017)


  \bibitem{wg4} {\sc C. Wang, X. Ye and S. Zhang},  {\em A Modified weak Galerkin finite element method for the Maxwell equations on polyhedral meshes}, Journal of Computational and Applied Mathematics, vol. 448, pp. 115918, 2024. 

  \bibitem{wgelas}{\sc  C. Wang and S. Zhang},  {\em Auto-stabilized weak Galerkin methods for elasticity problems},  Journal of Computational and Applied Mathematics, vol. 477, 117199, 2026.
  
 \bibitem{wg10}{\sc  C. Wang and S. Zhang},  {\em A Weak Galerkin Method for Elasticity Interface Problems}, Journal of Computational and Applied Mathematics, vol. 419, 114726, 2023. 
   
 \bibitem{wg16}{\sc  C. Wang and H. Zhou},  {\em A Weak Galerkin Finite Element Method for a Type of Fourth Order Problem arising from Fluorescence Tomography}, Journal of Scientific Computing, Vol. 71(3), pp. 897-918, 2017.  

 \bibitem{ye}{\sc  Y. Xiu and S. Zhang},  {\em  A stabilizer-free weak Galerkin finite element method on
polytopal meshes}, Journal of  Computational and Applied Mathematics, vol 371, 112699, 2020.

\bibitem{63}{\sc Yi, S.Y},  {\em Nonconforming mixed finite element methods for linear elasticity
using rectangular elements in two and three dimensions}. Calcolo 42(2),
115–133 (2005).


\bibitem{64}{\sc  Yi, S.Y},  {\em A new nonconforming mixed finite element method for linear
Elasticity}. Mathematical Models and Methods in Applied Sciences 16(07),
979–999 (2006).

\end{thebibliography}
\end{document}